\documentclass[12pt,reqno]{amsart}
\usepackage{amsthm}
\usepackage{amssymb}
\usepackage{graphics}
\usepackage{tikz}
\usetikzlibrary{shapes,backgrounds,calc}
\usepackage{latexsym}
\usepackage{multicol}
\usepackage{verbatim,enumerate}
\usepackage{accents}
\usepackage{cite}
\usepackage{multirow}
\usepackage{bigstrut}
\usepackage{amsthm}
\usepackage{amssymb}
\usepackage{graphics}
\usepackage{tikz}
\usetikzlibrary{shapes,backgrounds,calc}
\usepackage{latexsym}
\usepackage{multicol}
\usepackage{verbatim,enumerate}
\usepackage{accents}
\usepackage{cite}
\usepackage{array}
\usepackage[colorlinks=true, linkcolor=blue, citecolor=blue, urlcolor=black]{hyperref}
\usepackage{hyperref}
\usepackage{amsmath, amscd,url}
\usepackage{multirow}
\usepackage{bigstrut}
\usepackage{longtable}

\usepackage{stackengine}

\usepackage{hyperref}
\usepackage{amsmath, amscd,url}
\usepackage{longtable}

\advance\textwidth by 1.3in \advance\oddsidemargin by -.6in \advance\evensidemargin by -.6in
\theoremstyle{definition}

\newtheorem{theorem}{Theorem}[section]
\newtheorem{lemma}[theorem]{Lemma}

\theoremstyle{definition}
\newtheorem{definition}[theorem]{Definition}
\newtheorem{example}[theorem]{Example}

\theoremstyle{remark}
\newtheorem{remark}[theorem]{Remark}

\theoremstyle{definition}

\newcounter{cnt}
 \makeatletter
\def\mydggeometry{\makeatletter\dg@YGRID=1\dg@XGRID=20\unitlength=0.003pt\makeatother}
\makeatother \theoremstyle{remark}

\numberwithin{equation}{section}
\let\bwdg\bigwedge
\def\bigwedge{{\textstyle\bwdg}}

\newcommand{\nc}{\newcommand}
\newcommand{\rnc}{\renewcommand}

\nc{\cal}{\mathcal} \nc{\goth}{\mathfrak} \rnc{\bold}{\mathbf}

\nc\bomega{{\mbox{\boldmath $\omega$}}} \nc\bpsi{{\mbox{\boldmath $\Psi$}}}
 \nc\balpha{{\mbox{\boldmath $\alpha$}}}
 \nc\bpi{{\mbox{\boldmath $\pi$}}}
 \nc\bvpi{{\mbox{\boldmath $\varpi$}}}
\nc\chara{\operatorname{ch}}

  \nc\bxi{{\mbox{\boldmath $\xi$}}}
\nc\bmu{{\mbox{\boldmath $\mu$}}} \nc\bcN{{\mbox{\boldmath $\cal{N}$}}} \nc\bcm{{\mbox{\boldmath $\cal{M}$}}} \nc\blambda{{\mbox{\boldmath
$\lambda$}}}\nc\bnu{{\mbox{\boldmath $\nu$}}}

\makeatletter
\def\section{\def\@secnumfont{\mdseries}\@startsection{section}{1}%
  \z@{.7\linespacing\@plus\linespacing}{.5\linespacing}%
  {\normalfont\scshape\centering}}
\def\subsection{\def\@secnumfont{\bfseries}\@startsection{subsection}{2}%
  {\parindent}{.5\linespacing\@plus.7\linespacing}{-.5em}%
  {\normalfont\bfseries}}
\makeatother

 \nc{\Hom}{\operatorname{Hom}}
  \nc{\mode}{\operatorname{mod}}
\nc{\End}{\operatorname{End}} \nc{\wh}[1]{\widehat{#1}} \nc{\Ext}{\operatorname{Ext}} \nc{\ch}{\text{ch}} \nc{\ev}{\operatorname{ev}}
\nc{\Ob}{\operatorname{Ob}} \nc{\soc}{\operatorname{soc}} \nc{\rad}{\operatorname{rad}} \nc{\head}{\operatorname{head}}

 \nc{\Cal}{\cal} \nc{\Xp}[1]{X^+(#1)} \nc{\Xm}[1]{X^-(#1)}
\nc{\on}{\operatorname} \nc{\Z}{{\bold Z}} \nc{\J}{{\cal J}}  \nc{\Q}{{\bold Q}}

\nc{\N}{{\bold N}}  \nc\boa{\bold a} \nc\bob{\bold b} \nc\boc{\bold c} \nc\bod{\bold d} \nc\boe{\bold e} \nc\bof{\bold f} \nc\bog{\bold g}
\nc\boh{\bold h} \nc\boi{\bold i} \nc\boj{\bold j} \nc\bok{\bold k} \nc\bol{\bold l} \nc\bom{\bold m} \nc\bon{\mathbb n} \nc\boo{\bold o}
\nc\bop{\bold p} \nc\boq{\bold q} \nc\bor{\bold r} \nc\bos{\bold s} \nc\boT{\bold t} \nc\boF{\bold F} \nc\bou{\bold u} \nc\bov{\bold v}
\nc\bow{\bold w} \nc\boz{\bold z}\nc\ba{\bold A} \nc\bb{\bold B} \nc\bc{\mathbb C} \nc\bd{\bold D} \nc\be{\bold E} \nc\bg{\bold
G} \nc\bh{\bold H} \nc\bi{\bold I} \nc\bj{\bold J} \nc\bk{\bold K} \nc\bl{\bold L} \nc\bm{\bold M} \nc\bn{\mathbb N} \nc\bo{\bold O} \nc\bp{\bold
P} \nc\bq{\bold Q} \nc\br{\bold R} \nc\bs{\bold S} \nc\bt{\bold T} \nc\bu{\bold U} \nc\bv{\bold V} \nc\bw{\bold W} \nc\bz{\mathbb Z} \nc\bx{\bold
x} \nc\KR{\bold{KR}} \nc\rk{\bold{rk}} \nc\het{\text{ht }}

\nc\toa{\tilde a} \nc\tob{\tilde b} \nc\toc{\tilde c} \nc\tod{\tilde d} \nc\toe{\tilde e} \nc\tof{\tilde f} \nc\tog{\tilde g} \nc\toh{\tilde h}
\nc\toi{\tilde i} \nc\toj{\tilde j} \nc\tok{\tilde k} \nc\tol{\tilde l} \nc\tom{\tilde m} \nc\ton{\tilde n} \nc\too{\tilde o} \nc\toq{\tilde q}
\nc\tor{\tilde r} \nc\tos{\tilde s} \nc\toT{\tilde t} \nc\tou{\tilde u} \nc\tov{\tilde v} \nc\tow{\tilde w} \nc\toz{\tilde z} \nc\woi{w_{\omega_i}}

\begin{document}
%\setcounter{section}{0}
%\setcounter{tocdepth}{1}

%%%%%%%%%%%%%%%%%%%%%%%%%%%%%%%%%%%%%%%%%%%%

%\title{On Generalized $\phi$-Laguerre Polynomials and their irreducibility}
\title{On the irreducibility of extended Laguerre Polynomials}

\author[Anuj Jakhar]{Anuj Jakhar}
\author[Srinivas Kotyada]{Srinivas Kotyada}
\author[Arunabha Mukhopadhyay]{Arunabha Mukhopadhyay}
\address[Anuj Jakhar]{Department of Mathematics, Indian Institute of Technology (IIT) Madras}
\address[Srinivas Kotyada]{Department of Mathematics, The Institute of Mathematical Sciences (IMSc) Chennai}
\address[Arunabha Mukhopadhyay]{Department of Mathematics, The Institute of Mathematical Sciences (IMSc) Chennai}
%\address[Surender Kumar]{Department of Mathematics, Indian Institute of Technology (IIT) Bhilai}

\email[Anuj Jakhar]{anujjakhar@iitm.ac.in \\ anujiisermohali@gmail.com}
\email[Srinivas Kotyada]{srini@imsc.res.in}
\email[Arunabha Mukhopadhyay]{arunabham@imsc.res.in}

%\email[Surender Kumar]{surenderk@iitbhilai.ac.in}

%\thanks{The first author is employed at IIT Madras and thankful to SERB grant SRG/2021/000393.}

\subjclass [2010]{11C08, 11R04}
\keywords{Irreducibility, Polynomials, Newton Polygons}

\maketitle

\vspace{-0.2in}
%\begin{flushright}
	\hspace{2.4in}{\it{-- -- Dedicated to Professor Michael Filaseta}}
%\end{flushright}

\begin{abstract}
\noindent 
Let $m\geq 1$ and $a_m$ be integers. Let $\alpha$ be a rational number which is not a negative integer such that $\alpha = \frac{u}{v}$ with $\gcd(u,v) = 1, v>0$. Let $\phi(x)$ belonging to $\Z[x]$ be a monic polynomial which is irreducible modulo all the primes less than or equal to $vm+u$. Let $a_i(x)$ with $0\leq i\leq m-1$ belonging to $\Z[x]$  be polynomials having degree less than $\deg\phi(x)$. Assume that the content of $(a_ma_0(x))$ is not divisible by any prime less than or equal to $vm+u$. In this paper, we prove that the polynomials $L_{m,\alpha}^{\phi}(x) = \frac{1}{m!}(a_m\phi(x)^m+\sum\limits_{j=0}^{m-1}b_ja_j(x)\phi(x)^j)$ are irreducible over the rationals for all but finitely many $m$, where $b_j = \binom{m}{j}(m+\alpha)(m-1+\alpha)\cdots (j+1+\alpha)~~~\mbox{ for }0\leq j\leq m-1$. Further, we show that $L_{m,\alpha}^{\phi}(x)$ is irreducible over rationals for each $\alpha \in \{0, 1, 2, 3, 4\}$ unless $(m, \alpha) \in \{ (1,0), (2,2), (4,4),(6,4)\}.$ For proving our results, we use the notion of $\phi$-Newton polygon and some results from analytic number theory. We illustrate our results through examples. 

\end{abstract}
\maketitle

\section{Introduction and statements of results}\label{intro}

Let $m$ be a positive integer and $\alpha$ be an arbitrary real number.
  The generalized Laguerre polynomials are defined by 
 $$L_{m,\alpha}(x) = \sum\limits_{j=0}^{m}\frac{(m+\alpha)(m-1+\alpha)\cdots (j+1+\alpha)}{(m-j)!j!}(-x)^j.$$
 The classical Laguerre polynomials correspond to $\alpha = 0.$ They appear in various branches of Mathematics and Mathematical Physics. In a series of papers, I. Schur (cf. \cite{Sch1} -\cite{Sch3}) obtained irreducibility results for polynomials associated with $L_{m,\alpha}(x)$ by considering prime ideal factorization of the principal ideals generated by the coefficients of the polynomials. Later, Coleman \cite{Col} and Filaseta \cite{Fil} established Schur's results by invoking the main Newton polygon result of Dumas. It is interesting to note that Filaseta and Trifonov \cite{Fil2} showed the irreducibility of Laguerre polynomials corresponding to $\alpha = -2m-1$ which in turn confirm a conjecture of Grosswald on the irreducibility of Bessel polynomials.  The irreducibility of generalized Laguerre polynomials and their Galois groups were widely studied by many authors for several  values of $\alpha$ (cf.  \cite{Fila},\cite{Fila1},\cite{Haj},\cite{Haj1},\cite{Sho} -\cite{Sho5}).

 In this paper, let $m\geq 1$ and $a_m$ be integers. Let $\alpha$ be a rational number which is not a negative integer such that $\alpha = \frac{u}{v}$ with $\gcd(u,v) = 1, v>0$. Let $\phi(x)$ belonging to $\Z[x]$ be a monic polynomial which is irreducible modulo all the primes less than or equal to $vm+u$. Let $a_i(x)$ with $0\leq i\leq m-1$ belonging to $\Z[x]$  be polynomials having degree less than $\deg\phi(x)$. Assume that the content of $(a_ma_0(x))$ is not divisible by any prime less than or equal to $vm+u$.
Then, we define the generalized $\phi$-Laguerre polynomials by 
$$L_{m,\alpha}^{\phi}(x) = \frac{1}{m!}(a_m\phi(x)^m+\sum\limits_{j=0}^{m-1}b_ja_j(x)\phi(x)^j),$$ where $b_j = \binom{m}{j}(m+\alpha)(m-1+\alpha)\cdots (j+1+\alpha)~~~\mbox{ for }0\leq j\leq m-1.$

With this definition, the aim of this paper is to prove the following results.
\begin{theorem}\label{thm1.1}
	%There exists  a sufficiently large integer $m_0$ such that if $m > m_0$, then 
	The polynomials $L_{m,\alpha}^{\phi}(x)$ are irreducible over the rationals for all but finitely many positive integers $m$.  
\end{theorem}
%\begin{remark}
%	 We wish to point here that, it seems after going through the proof of Theorem \ref{thm1.1} that for $m > $ $L_{m,\alpha}^{\phi}(x)$ are irreducible over the rationals.
%\end{remark}

\begin{theorem}\label{thm1.2}
	The polynomials $L_{m,\alpha}^{\phi}(x)$ are irreducible over the rationals for $\alpha \in \{0, 1, 2, 3, 4\}$ unless $(m, \alpha) \in \{ (1,0), (2,2) , (4,4) ,(6,4)\} .$
\end{theorem}

\begin{remark}
	The above theorem gives rise to a wide class of irreducible polynomials whose irreducibility  do not seem to follow from the known irreducibility criterion (cf. \cite{brown2008}, \cite{JakBLMS}, \cite{JakAMS}, \cite{JakJA}, \cite{JakAM}, \cite{Ja-Sa} and \cite{Jho-Kh}).
\end{remark}
\begin{remark}
	Theorem \ref{thm1.2} can be extended for other values of $\alpha$, even for rational values of $\alpha.$
\end{remark}

%We wish to point here that for $m\geq 7$, the polynomials $L_{m,\alpha}^{\phi}(x)$ are irreducible over the rationals for $\alpha \in \{0, 1, 2, 3, 4\}$. 
%It may be noted that, in the classical case when $\phi(x) = x$, Theorem \ref{thm1.1} generalizes the main result of \cite{Fila1}. Theorem \ref{thm1.2} extends the main result of \cite[Theorem 1]{Ji-Kh}. 

It may be pointed out that in Theorem \ref{thm1.2}, the assumption  that ``the content of $(a_ma_0(x))$ is not divisible by any prime less than or equal to $vm+u$" cannot be dispensed with. This can be seen, for example taking $\phi(x)=x^2-x+17$ and $m=2$, as displayed in the tables below.   
\begin{longtable}[h!]{|m{1.0cm}|m{1cm}|m{1cm}|m{1cm}|m{5cm}|} 
\hline 
 $ \alpha $ & $a_2$ & $a_1(x) $ &$a_0 (x)$ & $L_{m,\alpha}^{\phi}(x)$ \\
			\hline
			1 & 3 & 1& -4 & $\frac{3}{2}(\phi(x)+4)(\phi(x)-2)$\\  \hline
			2 & 1 & 1 & -4 &$\frac{1}{2}(\phi(x)-4)(\phi(x)+12) $ \\ \hline
			3 & 1& 1 & -10 & $\frac{1}{2}(\phi(x)-10)(\phi(x)+20) $\\ \hline
			4 & 3& 1 & -6 & $ \frac{3}{2}(\phi(x)+10)(\phi(x)-6) $\\
\hline
 
 \caption{\label{tab:2} The content of $a_0(x)$ is divisible by $2$.}
\end{longtable}
\newpage
 \begin{longtable}[h!]{|m{1.0cm}|m{1cm}|m{1cm}|m{1cm}|m{5cm}|} 
\hline 
 $ \alpha $ & $a_2$ & $a_1(x) $ &$a_0 (x)$ & $L_{m,\alpha}^{\phi}(x)$ \\
			\hline
			1 & 6 & 2& 1 & $3(\phi(x)+1)^2$\\ \hline
			2 & 4 & 1 & -1 &$2(\phi(x)+3)(\phi(x)-1) $ \\ \hline
			3 & 2& 1 & -5 & $(\phi(x)+10)(\phi(x)-5) $\\ \hline
			4 & 18& 1 & -1 & $(\phi(x)-1)(9\phi(x)+15) $\\ \hline
 
 \caption{\label{tab:2} $a_2$ is divisible by $2$.}
\end{longtable}

We can also notice that Theorem \ref{thm1.2} may not hold if $a_n$ is replaced by a (monic) polynomial $a_n(x)$ with integer coefficients having degree less than $\deg\phi(x)$. Consider $\phi(x) = x^2-x+5$ which is irreducible modulo $2$ and $3$. Take $a_2(x) = x-3 = a_1(x) = a_0(x)$. Then the polynomial $\frac{1}{2}({a_2(x)}\phi(x)^2 + 2(2+\alpha)a_1(x)\phi(x)+(2+\alpha)(1+\alpha)a_0(x))$ has $3$ as a root for each $\alpha \in \{0,1,2,3,4\}$.

As an application, we now provide an example to show that Theorem \ref{thm1.2} leads to classes of irreducible polynomials. 
 \begin{example}
	Consider $\phi(x) = x^2-x+17$. It can be easily checked that $\phi(x)$ is irreducible modulo $2,3,5,7,11$ and $13$. Let $m$ and $\alpha$ be fixed integers such that $m\in \{3,5,7,8,9,10,11,12\}$, $\alpha \in \{0,1,2,3,4\}$. Let $a_m$ be an integer and $a_{i}(x) \in \Z[x]$ be polynomials each having degree less than $2$ for $0\leq i\leq m-1$. Assume that the content of $(a_ma_0(x))$ is not divisible by any prime less than or equal to $13$. Then by Theorem \ref{thm1.2}, the polynomial	$ L_{m,\alpha}^{\phi}(x)$ 
   is irreducible over $\Q$.
\end{example}

We use the concept of $\phi$-Newton polygon as introduced by Ore \cite{Ore}, few results on primes from analytical number theory to prove our theorems.% The proof of Theorems \ref{thm1.1}, \ref{thm1.2} is based on the different ideas starting with the concept of $\phi$-Newton polygon with respect to a prime number (see Definition \ref{def1}) . 

%\begin{remark}
%	It would not be difficult to extend Theorem \ref{thm1.2} further to other values of $\alpha$ and even consider rational values of $\alpha$ with these same methods.
%\end{remark}

\section{$\phi$-Newton polygon.}
In this section, we shall recall the notion of $\phi$-Newton polygon and state a result which will be used in the proof of the theorems. 

For a prime $p$, $v_p$ will denote the $p$-adic valuation of $\Q$ defined for any non-zero integer $b$ to be the highest power of $p$ dividing $b$. 
We shall denote by $v_p^x$ the Gaussian valuation extending $v_p$ defined on the polynomial ring $\Z[x]$ by $$v_p^x(\sum\limits_{i}b_ix^i) = \min_{i}\{v_p(b_i)\}, b_i \in \Z.$$ The $\phi$-Newton polygon with respect to $p$ is defined as below.
\begin{definition}\label{def1}
	Let $p$ be a prime number and $\phi(x) \in \Z[x]$ be a monic polynomial which is irreducible modulo $p$. Let $f(x)$ belonging to $\Z[x]$ be a polynomial having $\phi$-expansion\footnote{If $\phi(x)$ is a fixed monic polynomial with coefficients in $\Z$, then any $f(x)\in \Z[x]$ can be uniquely written as a finite sum $\sum\limits_{i}b_i(x)\phi(x)^i$ with $\deg b_i(x)< \deg \phi(x)$ for each $i$; this expansion will be referred to as the $\phi$-expansion of $f(x)$.} $\sum\limits_{i=0}^{n}b_i(x) \phi(x)^i$ with $b_0(x)b_n(x) \neq 0$. Let $P_i$ stand for the point in the plane having coordinates $(i, v_p^x(b_{n-i}(x)))$ when $b_{n-i}(x)\neq 0$, $0\leq i \leq n$.
Let $\mu_{ij}$ denote the slope of the line joining the points $P_i$ and $P_j$ if $b_{n-i}(x)b_{n-j}(x)\ne 0$. Let $i_1$ be the largest index $0< i_1 \leq n$ such that 
\begin{align*}
	\mu_{0i_1}=\min \{\mu_{0j}\ |\ 0<j \leq n,\ b_{n-j}(x)\ne 0\}.
\end{align*}
If $i_1<n$, let $i_2$ be the largest index $i_1< i_2 \leq n$ such that 
\begin{align*}
	\mu_{i_1i_2}=\min \{\mu_{i_1j}\ |\ i_1<j \leq n,\ b_{n-j}(x)\ne 0\}
\end{align*}
and so on. The $\phi$-Newton polygon of $f(x)$ with respect to $p$ is the polygonal path having segments $P_0P_{i_1}, P_{i_1}P_{i_2}, \dots, P_{i_{k-1}}P_{i_k}$ with $i_k=n$. These segments are called the edges of the $\phi$-Newton polygon of $f(x)$ and their slopes from left to right form a strictly increasing sequence. %The $\phi$-Newton polygon minus the horizontal part (if any) is called its principal part.
\end{definition}

We shall use the following lemma\cite[Theorem 1.3]{Ji-Kh} in the sequel. We omit its proof.
%With the above notation, we prove the following theorem, which was proved by Filaseta \cite[Lemma 2]{filaseta1995} in 1995 in the particular case when $\phi(x)=x$; the proof given here is similar to the one given by him.

\begin{lemma} \label{fila}
	Let $m, k$ and $\ell$ be integers with $0\leq \ell<k\leq \frac m2$ and $p$ be a prime. Let $\phi(x)\in \Z[x]$ be a monic polynomial which is irreducible modulo $p$. Let $f(x)$ belonging to $\Z[x]$ be a monic polynomial not divisible by $\phi(x)$ having $\phi$-expansion $\sum\limits_{i=0}^{m}f_i(x) \phi(x)^i$ with $f_m(x)\ne 0$. Assume that $v_p^x(f_i(x))> 0$ for $0\leq i\leq m-\ell-1$ and the right-most edge of the $\phi$-Newton polygon of $f(x)$ with respect to $p$ has slope less than $\frac 1k$. Let $a_0(x), a_1(x), \dots,a_m(x)$ be polynomials over $\Z$ satisfying the following conditions. 
	\begin{itemize}
		\item[(i)] $\deg a_i(x)< \deg \phi(x)-\deg f_i(x)$ for $0\leq i \leq m$, 
		\item[(ii)] $v_p^x(a_0(x))=0$, i.e., the content of $a_0(x)$ is not divisible by $p$,
		\item[(iii)] the leading coefficient of $a_m(x)$ is not divisible by $p$.
	\end{itemize}
	Then the polynomial $\sum\limits_{i=0}^{m} a_i(x) f_i(x) \phi(x)^i$	does not have a factor in $\Z[x]$ with degree lying in the interval $[(\ell+1)\deg \phi, (k+1)\deg \phi).$
\end{lemma}

\section{Proof of Theorem \ref{thm1.1}.}
Before writing the proof of the theorem, we first state two lemmas which are needed for the proof. We skip their proofs.

The following lemma is proved in \cite[Theorem 3]{BHP}. 
	\begin {lemma}\label{lm1}
	Let $a$ and $b$ be fixed relatively prime integers with $b>0$, and let $\pi(x; b, a)$ denote the number of primes $\leq x$ which are congruent to $a\pmod b$. If $b \leq (\log x)^A $ for some fixed $ A >0 $ and $x^{\frac{11}{20}+\varepsilon}\leq h\leq \frac{x}{\log x}$, then
	$$\pi(x;b, a) - \pi(x-h;b, a) \gg\frac{h}{\phi(b) \log x} ~\mbox{ for every }\varepsilon > 0.$$
\end{lemma}

The following lemma is an immediate consequence of Corollary 1.2 of \cite{Sh-Ti}. 
\begin{lemma}\label{lm2}
	Let $a,b,c$ and $d$ be integers with $bc-ad \neq 0$. Then the largest prime factor of $(am+b)(cm+d)$ tends to infinity as the integer $m$ tends to infinity.
\end{lemma}

\begin{proof}[Proof of Theorem \ref{thm1.1}]
	We fix an $\alpha \in \Q$ which is not a negative integer.  	It is enough to prove that $f(x) = a_m\phi(x)^m + \sum\limits_{j=0}^{m-1}b_ja_j(x)\phi(x)^j$ is irreducible over the rationals for sufficiently large $m$. Recall that $\alpha = \frac{u}{v}$ with $\gcd(u,v) = 1$ and $v>0.$

	We first show that $f(x)$ does not have a non-constant factor over $\Z$ with degree less than $\deg\phi(x)$. Let $p$ be a prime number dividing $vm+u$. Then by hypothesis, it follows that $p\nmid a_m$. %Since the leading coefficient of $f(x)$ is $a_m$, it is not divisible by $p$. 
	Let $c$ denote the content of $f(x)$. As $p\nmid a_m$, we have $p\nmid c$. Now suppose to the contrary that there exists a primitive non-constant polynomial $h(x) \in \Z[x]$ dividing $f(x)$ having degree less than $\deg \phi(x)$. Then in view of Gauss Lemma, there exists $g(x)\in \Z[x]$ such that $\frac{f(x)}{c} = h(x)g(x)$. The leading coefficient of $f(x)$ and hence those of $h(x)$ and $g(x)$ are coprime with $p$. Note that $p$ divides $b_j$ for $0\leq j\leq m-1$. Therefore on passing to $\Z/p\Z$, we see that the degree of $\bar{h}(x)$ is same as that of $h(x)$. Hence $\deg\bar{h}(x)$ is positive and less than $\deg\phi(x)$. This is impossible because $\bar{h}(x)$ is a divisor of $\frac{\overline{f}(x)}{\bar{c}} = \frac{\bar{a}_m}{\bar{c}}\bar{\phi}(x)^{m}$ and $\bar{\phi}(x)$ is irreducible over $\Z/p\Z$.

	Now using Lemma \ref{fila}, we shall show that for $k\in [1, \frac{m}{2}]$ and sufficiently large $m$, $f(x)$ can not have a factor in $\Z[x]$ with degree lying in the interval $[k\deg\phi(x), (k+1)\deg\phi(x))$. For using Lemma \ref{fila}, we consider the polynomial $g(x) = \sum\limits_{j=0}^{m}b_j\phi(x)^j$ with $b_m = a_m$. Keeping in mind that $\alpha$ is not a negative integer implies that for each $j \in \{0, 1, \cdots, m-1\}$, $m-j+\alpha$ and hence $v(m-j)+u$ cannot be zero. We assume that $g(x)$ has a factor in $\Z[x]$ with degree lying in the interval $[\deg\phi(x), (\frac{m}{2}+1)\deg \phi(x))$ and prove our theorem by obtaining a contradiction to Lemma \ref{fila}. 
	
	We divide the proof into cases depending on the size of $k$ with $1\leq k\leq m/2$.

\noindent Case (i): $m^{\frac{11}{20}+\varepsilon}<k\leq \frac{m}{2},$ where $\varepsilon >0$ is an arbitrary small constant.   
	 
	% Lemma \ref{lm1} implies that for relatively prime positive integers $a $ and $b$,
%	\begin{align*}
		%\pi(x;b, a) &= \frac{1}{\phi(b)}\int_2^x\frac{dt}{\log t} + O\bigg(\frac{x}{\log^{4} x}\bigg)\\&= \frac{1}{\phi(b)} \bigg(\frac{x}{\log x}+ \frac{x}{\log^{2} %x}+\frac{2x}{\log^{3} x}+ O \bigg(\frac{x}{\log^{4} x}\bigg)\bigg).
	%\end{align*}

By considering $\pi(x; b, a)- \pi(x-h; b, a)$, it follows from Lemma \ref{lm1} that for $a$ and $b$ fixed, the interval $[x-h, x)$ contains a prime congruent to $a$ modulo $b$ if $h \geq {x^{\frac{11}{20}+\varepsilon}}$ for all sufficiently large $x$. Taking $a = u,~b = v$ and $x = vm + u$, we deduce that for some integer $j \in[0, k)$, the number $v(m-j) + u$ is prime. Call such a prime $p$, and observe that $p \geq 2vm/3$ (since $v$ is a positive integer and $m$ is large). It follows that $p$ does not divide $v$. Observe that 

$$b_{\ell} = \binom{m}{\ell}\frac{(vm+u)(v(m-1)+u)\cdots (v(\ell+1)+u)}{v^{m-\ell}}~~~\mbox{ for }0\leq \ell\leq m-1.$$ %For $j\in\{0,1,.....k-1\}$, the numbers $v(m-j)+u$ appear in the numerator of the fraction on right-hand side above whenever $0\leq \ell\leq m-k$. 
Clearly
\begin{equation}\label{theq1}
	v_{p}(b_{\ell})\geq 1 ~~~\mbox{ for } 0\leq \ell\leq m-k.
\end{equation}
  The slope of the right-most edge of the $\phi$- Newton polygon of $g(x)$ with respect to $p$ is $$ \lambda:= \max\limits_{1 \leq j\leq m}\bigg( \frac{v_p(b_{0})- v_p(b_{j})}{j}\bigg).$$
Keeping in mind that $v_p(b_m) = 0$, to obtain a contradiction from Lemma \ref{fila}  for the case under consideration, we need to show that $\lambda < \frac{1}{k}.$ For this purpose, it is enough to show that $v_{p}(b_{0})=1$, because it follows from (\ref{theq1}), the fact $k\leq m-k$ that $v_{p}(b_{j})\geq 1 > 1-(j/k)$ for $1\leq j\leq k$, and clearly $ v_p(b_{j}) > 1-(j/k)$ for $k <j \leq m$. 

 To show $v_p(b_0) = 1$, it can be easily checked using the fact $ p\geq 2vm/3$ and $m$  sufficiently large that $$ 2p> vm+u\geq v(m-j) + u \geq v+u>-p ~~\mbox{ for }~~0\leq j\leq m-1.$$ As indicated earlier, none of the $v(m-j)+u$ can be zero. Hence, $ p$ itself is the only multiple of $p$ among the number $v(m-j)+u$ with $ 0\leq j\leq m-1$. Since $p\nmid v$ and $b_{0} =((vm+u)(v(m-1)+u)\cdots (v+u))/v^{m}$, we obtain $v_{p}(b_{0})=1$. Hence, for $m$ sufficiently large,  $f(x)$ can not have a factor in $\Z[x]$ with degree lying in the interval $ [k\deg\phi(x), (k+1)\deg\phi(x))$ for $ m^{\frac{11}{20}+\varepsilon}<k\leq \frac{m}{2}$.\\
 %\end{case}
%\begin{case}

\noindent Case (ii): $k_{0}\leq k\leq m^{\frac{11}{20}+\varepsilon}$ with $k_{0} = k_{0}(u,v)$ a sufficiently large integer.\newline\newline We wish to show that $f(x)$ does not have a factor in $\Z[x]$ with degree lying in the interval $[k\deg\phi(x), (k+1)\deg\phi(x))$ for $k \in [k_0, m^{\frac{11}{20}+\epsilon}]$. 

Let $z= k\sqrt{\log k}.$ 
We claim that there  is a prime $p>z$ that divides $v(m-j)+u$ for some $j\in\{0, 1,2....,k-1\}$. Then $(\ref{theq1})$ follows as before, and we will obtain a contradiction to Lemma \ref{fila} by showing $v_p(b_{j}) > v_p(b_{0})-(j/k)$ for $ 1\leq j\leq m$ as this implies the slope of the right-most edge of the $\phi$-Newton polygon of $g(x)$ with respect to $p$ has slope $<1/k$.

 Let \vspace{-0.1in}
$$  T=\{v(m-j)+u: 0\leq j\leq k-1\}.$$
Since $m$ is large, we deduce that each elements of $T$ is greater than or equal to $m/2$. Also, observe that $\gcd(u,v)=1$ implies that each element of $T$ is relatively prime to $v$. For each prime $p\leq z$, we consider an element $a_{p} \in {T}$ with $v_{p}(a_{p})$ as large as possible. Then we consider the set  
$$ S=T-\{ a_{p} : p\nmid v,~ p\leq z\}.$$ Clearly $|S|\geq k-\pi(z).$   By the well-known Chebyshev bound, we have $|S| \geq k-  \frac{2k}{\sqrt{\log k}}.$  

Consider the prime $ p\leq z$ with $p$ not dividing $v$, and let $r=v_{p}(a_{p}).$ By definition of $ a_{p}$, if $j>r$, then there are no multiples of $p^{j}$ in $T$ (and, hence, in $S$). For $1\leq j\leq r$, there are less than or equal to $[k/p^{j}]+1$ multiples of $p^{j}$ in $T$ and, hence, at most $[k/p^{j}]$ multiples of $p^{j}$ in $S$. Therefore,
 $$ v_{p}\bigg(\prod\limits_{s\in S}s\bigg)\leq\sum\limits_{j=1}^{r}\bigg[\frac{k}{p^{j}} \bigg] \leq v_{p}(k!),$$ and hence 
 \begin{equation}\label{eq1.1a}
 	\prod\limits_{s\in S}\prod\limits_{p\leq z} p^{v_{p}(s)}\leq k!\leq k^{k}.
 \end{equation}
 On the other hand using the fact that $k\leq m^{\frac{11}{20}+\varepsilon}$, we see that
 $$\prod\limits_{s\in S} s \geq  \bigg(\frac{m}{2}\bigg)^{|S|}\geq \bigg(\frac{k^{\frac{29}{20}-\varepsilon}}{2}\bigg)^{|S|}.$$ 
 Recalling our bound on $|S|$, we obtain
 \begin{align*}
 	\log\bigg(\prod\limits_{s\in S}{ s}\bigg)&\geq \bigg( k-\frac{2k}{\sqrt{\log k}}\bigg)\bigg(\bigg(\frac{29}{20}-\varepsilon\bigg)\log k-\log2\bigg)\\
 	&\geq k\log k+ \bigg(\frac{9}{20}-\varepsilon\bigg)k \log k+ O(k\sqrt{\log k}).
 \end{align*}
Since $k\geq k_{0}$ where $k_{0}$ is sufficiently large and using (\ref{eq1.1a}), we have $$  \log\bigg(\prod\limits_{s\in S}s \bigg) > k\log k \geq \log\bigg(\prod\limits_{s\in S}\prod\limits_{p\leq z} p^{v_{p}(s)}\bigg).$$ It follows that there is a prime $p>z$ that divides some element of $S$ and, hence, divides some element of $T$. This proves our claim.

%Now we shall show $v_p(b_{j}) > v_p(b_{0})-(j/k)$ for $ 1\leq j\leq m$.
 Fix a prime $p>z$ that divides an element in $T$. Fix $j\in\{1,2,...,m\}$. It only remains to show that $v_p(b_{0})- v_p(b_{j}) < \frac{j}{k}$. Observe that
 \begin{align*}
 	v_p(b_{0})-v_p(b_{j})&\leq v_p((vj+u)(v(j-1)+u)\cdots (v+u))\\
 	&\leq v_p((vj+|u|)!)< \frac{vj+|u|}{p-1}.
 \end{align*}
Since $p>z=k\sqrt{\log k}$ and $k \geq k_{0},$ we deduce that $(vj+|u|)/(p-1)<j/k$ and the inequality $v_p(b_{0})- v_p(b_{j}) < \frac{j}{k}$ follows. Hence, as indicated at the beginning of this case, we obtain a contradiction to Lemma \ref{fila}. \\
%\end{case}
%\begin{case}
	
%\end{case}

\noindent Case (iii): {$2\leq k <k_{0}.$}\newline
By Lemma \ref{lm2} (with $a=v,~ b=u, ~c=v,$ and $d=u-v$), the largest prime factor of the product $(vm+u)(v(m-1)+u)$ tends to infinity. Since $m$ is large, we deduce that there is a prime $p>(v+|u|)k_{0}$ that divides $(vm+u)(v(m-1)+u).$
The argument now follows as in the previous case. In particular,
$$ \frac{v_p(b_{0})-v_p(b_{j})}{j} < \frac{vj+|u|}{j(p-1)}\leq\frac{v+|u|}{p-1}\leq\frac{1}{k_{0}}<\frac{1}{k} ~\mbox{ for } 1\leq j\leq m.$$
Hence, in this case, we also obtain a contradiction.\\

\noindent Case (iv): $k=1.$

 It only remain to prove that $f(x)$ does not have a factor in $\Z[x]$ with degree lying in the interval $[\deg\phi(x), 2\deg\phi(x))$. This will be achieved once we show that there exists  a prime $p > v+|u|$ such that $p|b_j$ for $0\leq j\leq m-1$. If $u=0$, then it is always true for $m\geq 2$. Hence we can assume that $u\neq 0$. In this situation, from Lemma \ref{lm2}, the largest prime factor of $m(vm+u)$ tends to infinity with $m$. We consider a large prime factor $p$ of this product. Note this implies $p\nmid v$. As in the previous case, we are through if $p$ divides $vm+u$. So suppose $p|m$. The binomial coefficient $\binom{m}{j}$ appears in the definition of $b_j$, and this is sufficient to guarantee that $v_p(b_j)\geq 1$ and $v_p(b_{m-j}) \geq 1$ for $1\leq j\leq p-1$. On the other hand,
$$b_j = \binom{m}{j}\frac{(vm+u)(v(m-1)+u)\cdots (v(j+1)+u)}{v^{m-j}}.$$
For $j\leq m-p$, the numerator of the fraction on the right is a product of $\geq p$ consecutive terms in the arithmetic progression $vt+u$ with $\gcd(p,v) =1$; thus, $v_p(b_{m-j}) \geq 1$ for $j\geq p$. This implies that $(\ref{theq1})$ holds with $k=1$. It follows, along the lines of the previous two cases, that $v_p(b_{0})- v_p(b_{j}) < \frac{j}{k}$ for $1\leq j\leq m$. A contradiction to Lemma \ref{fila} is again obtained. 

Therefore combining Cases (i)-(iv), we have shown that for $k\in [1, \frac{m}{2}]$ and sufficiently large $m$, $f(x)$ can not have a factor in $\Z[x]$ with degree lying in the interval $[k\deg\phi(x), (k+1)\deg\phi(x))$. This completes the proof of the theorem.
\end{proof}

	\section{Proof of Theorem \ref{thm1.2}}
	We shall first state two results which will be required for the proof  of our theorem.

	\begin{lemma}\label{prime7}
		For integers $m$ and $k$ with $m\geq k\geq 2$, let $P(m,k)$ denote the product $(m+1)(m+2)\cdots (m+k)$. For $t\geq 1$, define $S_t$ as the set of pairs $(m,k)$ for which $m\geq k\geq 2$ and the largest prime factor of $P(m,k)$ is $k+t$. Then 
		$$S_1 = \{(2,2), (7,2)\},~~~S_2 = \{(3,3), (7,3), (5,5)\},$$
		$$S_3 = \{(3,2), (4,2), (8,2), (14,2), (23, 2), (79,2), (4,4), (5,4), (6,4)\},$$ and
		$$S_4 = \{(4,3), (5,3), (6,3), (13,3), (47,3)\}.$$
	\end{lemma}
	It may be pointed out that, in the above lemma, for every $(m, k)\not\in S_1\cup S_2\cup \cdots \cup S_{t-1}$ with $m\geq k\geq 2$, the largest prime factor of $P(m,k)$ is $\geq k+t.$
	
	The proof of the above lemma relies on a method due to Lehmer \cite{Leh}(for classifying all the cases where $(m+1)(m+2)$ has all its prime factors bounded below by a prescribed bound) and a result of Ecklund et. al. \cite{EEE}.
	
	The next lemma is on solutions of some equations. The assertion $(i)$ is a special case of Catalan's conjecture, now Mihailescu's Theorem when $r>1, s>1$, see \cite{Mih04}. The case $r=1$ or $s=1$ is immediate. The assertions $(ii)$ and $(iii)$ are due to Nagell \cite{Nag58}. For assertions $(iv)-(vi)$, see \cite[Lemma 4]{LaSg06a}.
	\begin{lemma}\label{dio}
		Let $r>0, s>0, t>0$ be integers. The solutions of the following equations are given by 
		%\begin{center}
\begin{longtable}[h!]{m{1.0cm} m{5.4cm} m{7.3cm}} 
\hline 
   & Equations  &Solutions \\ \hline
 $(i)$&$a^r-b^s=\pm 1,a,b\in\{2,3,5\}$ & $3-2=1,2^2-3=1,5-2^2=1,3^2-2^3=1$ \\ \hline
 $(ii)$&$2^r+3^s=5^t$ & $2+3=5,2^4+3^2=5^2$ \\ \hline
 $(iii)$&$2^r+5^s=3^t$ & $2+1=3, 2 +25 = 27, 4+5 = 9,8+1 = 9$ \\ \hline
 $(iv)$&$2^r3^s-5^t = \pm 1$ & $2\cdot 3-5=1, 2^3\cdot 3-5^2 = -1$  \\ \hline
 $(v)$&$3^r5^s-2^t = \pm 1$ & $3\cdot 5-2^4 = -1$ \\ \hline
 $(vi)$&$2^r5^s-3^t = \pm 1$ & $2\cdot 5-3^2 = 1, 2^4\cdot 5-3^4 = -1$ \\ \hline
 %\caption{\label{tab:2}The values of $m_i$'s for $1\leq i\leq 7$ depending on $c$ and $k$.}
\end{longtable}
%\end{center}     
	\end{lemma}
	
	\begin{proof}[Proof of Theorem \ref{thm1.2}] 
		Since $\alpha\in\{0,1,2,3,4\}$, we have $u\in\{0,1,2,3,4\}$ and $v=1$.  We define  $f(x)$
	and $g(x)$ as we did in the proof of Theorem \ref{thm1.1}. As there always exists a prime $p$ dividing $m+u$ except $(m, u) = (1,0)$, it follows from the second paragraph of the proof of Theorem \ref{thm1.1} that $f(x)$ can not have a non-constant factor over $\Z$ with degree less than $\deg\phi(x).$

	Therefore, it is enough to prove that $f(x)$ does not have a factor in $\Z[x]$ with degree lying in the interval $[k\deg\phi(x), (k+1)\deg\phi(x))$ with $k\in [1,\frac{m}{2}]$ and $m\geq 2$. 
	 %In particular $M$ induces the positive integers $\leq 100.$ Our approach is to obtain a contradiction  to lemma\ref{lm2} when $m\notin M.$ For $m\in M,$ we will rely on computations to verify that $L^{\alpha}_{m}(x)$ irreducible over the rationals in these cases. We do not address the irreducibility  of more general polynomials $f(x)$ for $m\in M.$ 
	 We divide our remaining proof into two cases depending on whether $k\geq 2$ or $k=1.$
	 
	 Suppose  first that  $k\geq 2$ and $m\geq 2.$ In the notaion of Lemma \ref{prime7}, observe that $P(m-k+u, k)$ divides $b_{j}$ for $0\leq j\leq m-k.$ Since $m-k+u\geq m-k\geq k,$ Lemma \ref{prime7} implies that $P(m-k+u, k)$
	has a prime divisor $p\geq k+u+1$ unless
	\begin{equation}\label{anuj}
		(m-k+u,k)\in S = S_{1} \cup S_{2}\cup \cdots \cup S_{u}, 
	\end{equation}   
	where $S_i$ is as given in Lemma \ref{prime7} for $1\leq i\leq 4$.
	  We first assume that $ (m-k+u,~k)\notin S$ and fix a prime $p$ with $p\geq k+u+1$ as above.  Since $p$ divides $P(m-k+u, k)$, we have $p$ divides $b_{j}$ for $0\leq j\leq m-k.$ We wish to point here that if there exists a prime $p\geq u+2$ with $p|b_j$ with $0\leq j\leq m-1$, then all the following arguments will work even for $k= 1.$ By hypothesis $p\nmid b_{m}$. For using Lemma \ref{fila}, we are left with verifying that the slope of the right-most edge of the $\phi$-Newton polygon of $g(x)$ with respect to $p$ has slope $< \frac{1}{k}$ for $k\geq 2$.
	  ~~~~The slope of the right-most edge is given by
	 $$ \max\limits_{1\leq j\leq m}\bigg\{\frac{v_p(b_{0})-v_p(b_{j})}{j}\bigg\}.$$ Observe that 
	 \begin{equation}
	 	\frac{b_{0}}{b_{j}} = \frac{(j+ u)(j-1+ u) \cdots (1+u)}{\binom{m}{j}}.
	 \end{equation}
	 It follows that
	 $$ v_p(b_{0})-v_p(b_{j})\leq v_p((j+u)(j-1+u)\cdots (1+u))\leq v_p((j+u)!).$$ If $j\leq k,$ then $p\geq k+u+1$ implies that $v_p(b_{0})-v_p(b_{j})\leq 0.$ If $j> k,$ then using the  simple inequality $(j+\alpha)/(k+\alpha) < j/k $, we obtain
	 $$ v_p(b_{0})-v_p(b_{j})\leq v_p((j+\alpha)!) <\frac{j+\alpha}{p-1}\leq \frac{j+\alpha}{k+\alpha}<\frac{j}{k}.$$ We combine the above to deduce, as desired, that the right-most edge has slope $<\frac{1}{k}.$ Hence, using Lemma \ref{fila}, $f(x)$ can not a factor in $\Z[x]$ with degree lying in the interval $[k\deg\phi(x), (k+1)\deg\phi(x))$ with $k\in [2,\frac{m}{2}]$, unless $(m-k+u, k) \in S.$ Now suppose $(m-k+u,k) \in S$. The following table provides us all the exceptional cases corresponding to each $u$ with $2k\leq m$.

	 \begin{longtable}[h!]{|m{1.0cm}|m{9.7cm}|} 
\hline 
    $u$  & $(m,k)$ \\ \hline
 1 & $(8,2)$ \\ \hline
 2 & $(7,2),(8,3)$ \\ \hline
 3 & $(6,2),(7,3),(7,2),(13,2),(22,2),(78,2)$ \\ \hline
 4 & $(5,2),(6,2),(12,2),(21,2),(77,2),(6,3), (12,3),(46,3)$ \\ \hline
 
 %\caption{\label{tab:2}The values of $m_i$'s for $1\leq i\leq 7$ depending on $c$ and $k$.}
\end{longtable}
	  For all these pairs $(m,k)$ except for $(m, k, u)\in \{(6,2,4), (6,3,4)\}$, we now  provide a prime number $p$ (see the following tables) such that $p$ divides $b_j$ for $0\leq j\leq m-k$ and the right-most edge has slope $<\frac{1}{k},$ and hence we are done for all these cases in view of Lemma \ref{fila}.
	  	 
	    \begin{minipage}[]{0.30\linewidth}
        \begin{center}
\begin{longtable}[h!]{|m{0.7cm}|m{0.7cm}|m{0.7cm}|m{0.7cm}|} 
\hline 
  $u$ & $k$ & $m$   &$p$ \\ \hline
 $1$ & $2$ &$8$ & $7$ \\ \hline
 $2$ &$2$ &$7$ & $7$ \\ \hline
 $2$ & $3$&$8$ & $7$  \\ \hline
 $3$ & $2$ &$6$ & $5$ \\ \hline
 $3$ & $3$ & $7$   &$7$ \\ \hline
 $3$ & $2$ &$7$ & $7$ \\ \hline
 %\caption{\label{tab:2}The values of $m_i$'s for $1\leq i\leq 7$ depending on $c$ and $k$.}
\end{longtable}
\end{center}      
    \end{minipage}
    \hfill
    \begin{minipage}[]{0.30\linewidth}
    \begin{center}
\begin{longtable}[h!]{|m{0.7cm}|m{0.7cm}|m{0.7cm}|m{0.7cm}|} 
\hline 
    $u$ & $k$ & $m$   &$p$ \\ \hline
 $3$ & $2$ & $13$ & $13$ \\ \hline
 $3$ &$2$ & $22$ & $7$ \\ \hline
 $3$ & $2$& $78$ & $7$  \\ \hline
 $4$ & $2$ & $5$ & $11$ \\ \hline
 $4$ & $2$ & $6$   &$--$ \\ \hline
 $4$ & $2$ &$12$ & $11$ \\ \hline
 %\caption{\label{tab:2}The values of $m_i$'s for $1\leq i\leq 7$ depending on $c$ and $k$.}
\end{longtable}
\end{center}     
         \end{minipage}
             \hfill
    \begin{minipage}[]{0.23\linewidth}
    \begin{center}
\begin{longtable}[h!]{|m{0.7cm}|m{0.7cm}|m{0.7cm}|m{0.7cm}|} 
\hline 
    $u$ & $k$ & $m$   & $p$ \\ \hline
 $4$ & $2$ &$21$ & $7$ \\ \hline
 $4$ &$2$ &$77$ & $7$ \\ \hline
 $4$ & $3$&$6$ & $--$  \\ \hline
 $4$ & $3$ &$12$ & $11$ \\ \hline
 $4$ & $3$ & $46$   &$23$ \\ \hline
 %\caption{\label{tab:2}The values of $m_i$'s for $1\leq i\leq 7$ depending on $c$ and $k$.}
\end{longtable}
\end{center}     
         \end{minipage}
\\

\indent Now we deal with the case when $k = 1$, $m\geq 2$. We divide this case into five sub-cases according to values of $u$. We first recall, as pointed out in the previous case, that if there exists a prime $p \geq u+2$ such that $p|b_j$ for $0\leq j\leq m-1$, then we are done in this case in view of Lemma \ref{fila}.

\noindent Subcase (i) $u=0$. In this situation, note that $k+u+1 = 2$ and there is clearly a prime $p \geq 2$ dividing $m$ for $m\geq 2$ and hence, in view of Lemma \ref{fila}, $f(x)$ can not have a factor in $\Z[x]$ with degree lying in the interval $[\deg\phi(x), 2\deg\phi(x))$. So, we are done.

\noindent Subcase (ii) $u=1$. This subcase follows by observing that there always exists a prime $p \geq 3(=u+2)$ dividing $m(m+1)$ for $m\geq 2$. 

\noindent Subcase (iii) $u=2$.  Observe that $m(m+2)$ always divides $b_j$ for $0\leq j\leq m-1$ for $m\geq 3$. Note that $m=2$ is an exceptional case. Keeping in mind Lemma \ref{dio}, one can easily verify that there always exists a prime $p \geq 5$ dividing $m(m+2)$ unless $m \in \{4,6,16\}$. 

Let $m \in \{4,6,16\}$. We set $p$ = $2$ or $3$ according as $m\in \{4, 16\}$ or $m = 6$. One can easily check that $p$ divides $b_j$ for $0\leq j\leq m-1$. Keeping in mind Equation (\ref{anuj}), one can check that the right-most edge of the $\phi$-Newton polygon of $g(x)$ with respect to $p$ has slope $< 1$. Thus we are done in view of Lemma \ref{fila}.

\noindent Subcase (iv) $u=3$. Note that $m(m+3)$ always divides $b_j$ for $0\leq j\leq m-1$ for $m\geq 2$. Keeping in mind Lemma \ref{dio}, it can be easily seen that there always exists a prime $p \geq 5$ dividing $m(m+3)$ unless $m \in \{3, 6, 9, 24\}$. 

If $m \in \{3, 6, 9, 24\}$, then one can check that $3$ divides $b_j$ for $0\leq j\leq m-1$ and the right-most edge of the $\phi$-Newton polygon of $g(x)$ with respect to $3$ has slope $< 1$. So, using again Lemma \ref{fila}, we are done.

\noindent Subcase (v) $u=4$. Using Lemma \ref{dio}, it can be verified that there always exists a prime $p \geq 7$ dividing $m(m+4)$, and hence $b_j$, for $0\leq j\leq m-1$ unless $m \in \{2, 4, 5, 8, 12, 16, 20, 32, 36, 60, 96, 320\}.$ 

Consider $p = 3$ if $m\in \{2, 8, 12, 32, 36, 96\}$, $p=2$ if $m = 16$, and $p=5$ if $m\in \{5, 20, 60, 320\}.$ One can check that $p$ divides $b_j$ for $0\leq j\leq m-1$ and the right-most edge of the $\phi$-Newton polygon of $g(x)$ with respect to $p$ has slope $< 1$. Therefore, using Lemma \ref{fila}, we are done.

This completes the proof of the theorem.
	\end{proof}

 \medskip
  \vspace{-3mm}
 % \newpage

 \end{document}